\documentclass[letterpaper,12pt]{amsart}
\usepackage{hyperref}
\usepackage{latexsym}
\usepackage{amssymb}
\usepackage[cp850]{inputenc}
\usepackage{epsfig}
\usepackage{amsthm}
\usepackage{amscd}
\usepackage{amsmath}
\usepackage{amsfonts}
\usepackage{graphics}
\usepackage[all]{xy}

\newtheorem{theorem}{Theorem}
\newtheorem{lemma}[theorem]{Lemma}
\newtheorem{corollary}[theorem]{Corollary}
\newtheorem{prop}[theorem]{Proposition}

\newtheorem*{theorem*}{Theorem}
\newtheorem*{corollary*}{Corollary}
\theoremstyle{definition}

\newtheorem*{remark*}{Remark}

\newtheorem*{definition*}{Definition}
\newtheorem*{example*}{Example}

\numberwithin{theorem}{section}

\newcommand{\BR}{\mathbb R} 
\newcommand{\BN}{\mathbb N} \newcommand{\BQ}{\mathbb Q}
 \newcommand{\BZ}{\mathbb Z}

\newcommand{\aut}{\textup{Aut}(F_n)}

\newcommand{\wh}{\widehat}
\newcommand{\nid}{\noindent}

\newcommand{\mP}{\mathcal{P}}
\newcommand{\mC}{\textup{span }\wh{\mathcal{C}(\underline{\Gamma})}}

\DeclareMathOperator{\shd}{\mathcal{S}}
\DeclareMathOperator{\oshd}{\widetilde{\shd}}
\DeclareMathOperator{\drk}{\mathfrak{D}}

\newcommand{\comment}[1]{}

\begin{document}
\title    {Homological shadows of attracting laminations}
\author   {Asaf Hadari}
\date{\today}
\begin{abstract}
Given a free group $F_n$,  a fully irreducible automorphism $f \in \aut$, and a generic element $x \in F_n$, the elements $f^k(x)$ converge in the appropriate sense to an object called an attracting lamination of $f$. When the action of $f$  on $\frac{F_n}{[F_n, F_n]}$ has finite order, we introduce a homological version of this convergence, in which the attracting object is a convex polytope with rational vertices, together with a  measure supported at a point with algebraic coordinates.  \end{abstract}
\maketitle

\vspace {10mm}
\section{Introduction}

Let $F_n = \langle a_1, \ldots, a_n \rangle $ be the free group on $n$ generators. Any automorphism  $f: F_n \to F_n$ induces an automorphism $f_{ab}$ of $H_1(F_n, \BZ) \cong \frac{F_n}{[F_n,F_n]} \cong \BZ^n$. This gives a natural map $\aut \to GL_n(\BZ)$In this paper we attach a coarse-geometric, homological invariant to any fully irreducible automorphism $f$ such that $f_{ab}$ has finite order.  This invariant addresses the following questions.

\begin{enumerate}
\item Given a word of the form $w = f^N(x)$, where $x \in F_n$ and $N \gg 0$, what are the images in $\BZ^n$ of \emph{subwords} of $w$? (Where a subword of $w$ is an element $b \in F_n$ such that $w = abc$ for some $a,c \in F_n$, and this  product is reduced. )
\item If a subword of $w$ is chosen at random (for a suitable definition of random), what should we expect its image in $\BZ^n$ to be?
\item How do the answers to the above questions depend on $N,x$ and the choice of generating set for $F_n$?
\end{enumerate}

\subsection{Two  examples.}
The following automorphisms showcase the kinds of invariants we are interested in. Let $F _3 = \langle a,b,c \rangle$ be the free group on three generators. Consider the automorphism $f$ given by $f(x) = bxb^{-1}$. Note that $f_{ab} = I_3$. Given any $N$, we have that $w = f^N(a) = b^N a b^{-N}$. Restricting ourselves to subwords of $w$ obtained by taking all letters from the first to the $j$-th, for some $j$,  we have that the possible images of such elements are exactly those vectors of the form:
$$\bigcup _{i = 0}^N \{ \left(\begin{array}{c} 0 \\ i \\ 0 \end{array}\right), \left(\begin{array}{c} 1 \\ N - i \\ 0 \end{array}\right) \} $$

 Coarsely speaking, this set forms a ray in $\BR^3$, namely the positive half of the $y$-axis. Now suppose that we choose a random subword of $w$ using the following model: choose a number $i$ uniformly from $1$ to $2N+1$, and then take the subword of $w$ beginning at the first letter and ending at $i$-th letter. Under this model, all of the elements of the set described above are equally likely. Coarsely, we get the Lebesgue measure on the ray.  \\

\nid Now, consider the more complicated automorphism given by:

$$\begin{array}{l}g(a) = cbc^{-1}bab^{-1}cb^{-1}c^{-1} \\ g(b) = cbc^{-1} \\ g(c) = cbab^{-1}cbc^{-1}ba^{-1}b^{-1}cb^{-1}c^{-1} \end{array} $$

\nid Once again, we have that $g_{ab} = I_3$. However, in this case, it is very difficult to find an explicit expression for $w = g^N(a)$. The figures below display what we will call the shadows of $g(a), \ldots, g^6(a)$ (arranged from left to right). In each, we show the images in $\BR^3$ of all subwords beginning at the first letter, and we connect two such images with a segment if they corresponds to subwords, one of which is attained from the other by adding one letter.

\medskip
\includegraphics[width=20mm]{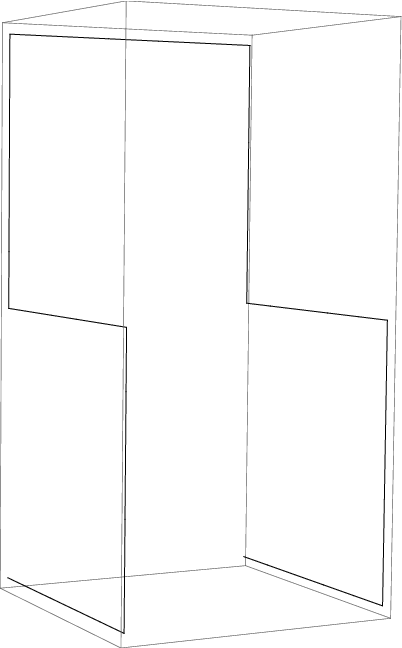} \includegraphics[width=20mm]{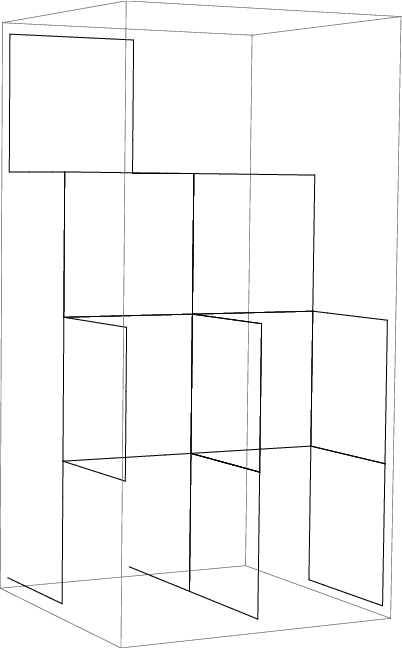} \includegraphics[width=20mm]{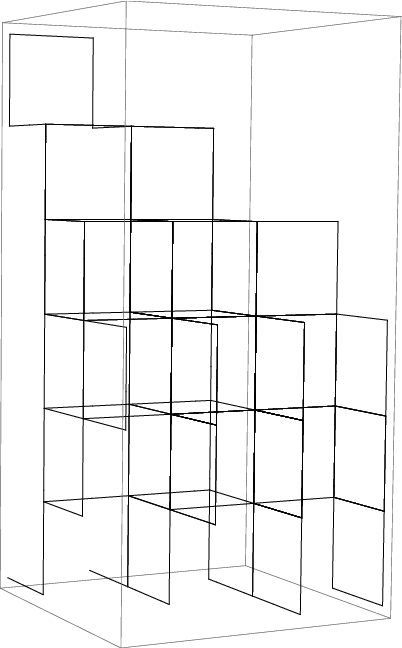} \includegraphics[width=20mm]{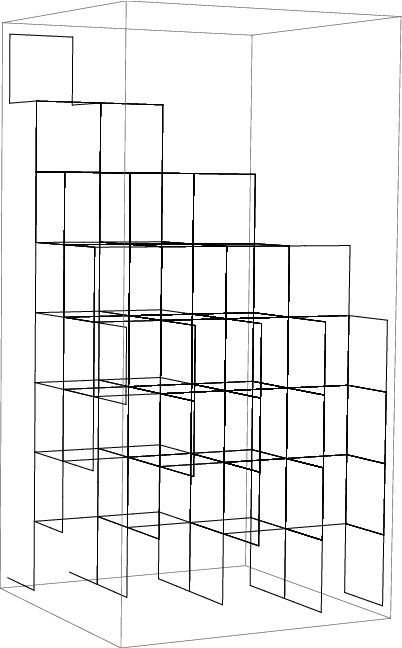} \includegraphics[width=20mm]{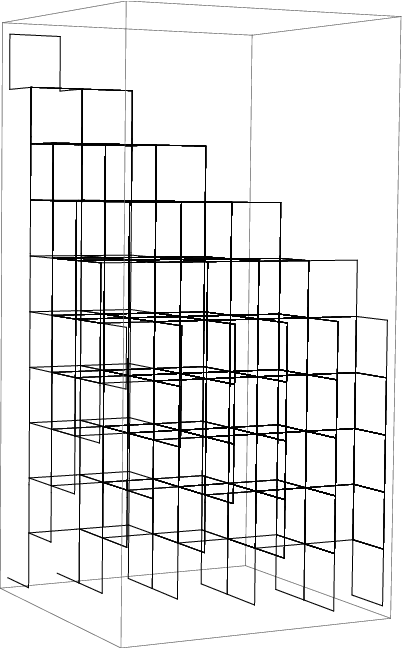} \includegraphics[width=20mm]{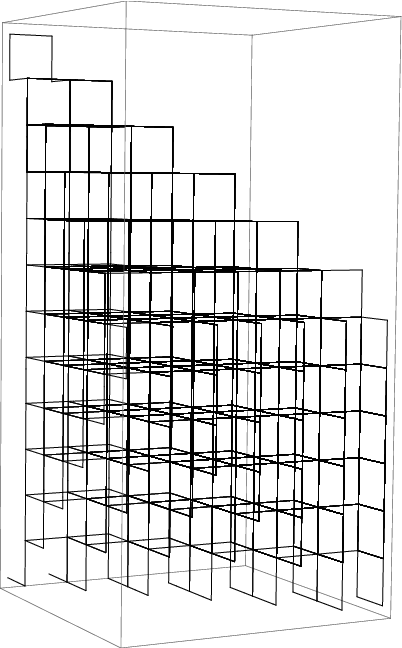}      \\

\nid Looking at the figures, a clear pattern seems to emerge. Our goal in this paper is to describe this pattern for automorphisms of this type.

\subsection{Shadows and darkness}
In this section we give the definitions necessary to formalize the above discussion, and to state our Theorems. Let $G$ be a graph with $N$ edges. Number the edges $d_1, \ldots, d_N$. Let $l: E(G) \to \BR_{> 0}$ be a function that assigns to each edge a positive length. Choose an identification of the $i$-th edge with the interval $[0, l(d_i)]$. Note that this identification also gives us a preferred orientation on each edge. \\

 To any path $p: [0,1] \to G$ in the graph, we associate a path $\oshd p: [0,1] \to \BR^N$ in the following way.  Define a form  $\delta$ on $G$ by setting by setting $\delta|_{V(G)} = 0 dt$ and $\delta | _{d_i} = \frac{1}{l(d_i)} e_i dt$ where $e_i$ is the $i$-th vector of the standard basis for $\BR^N$. Let $\oshd p (t) = \int_{0}^t p^*\delta$ and let $\shd p = \textup{Im} \oshd p$. We call $\shd p$ the \emph{shadow of p}, and $\oshd p$ the \emph{parametrized shadow of $p$}.   For every integer $k$, we set $\oshd_k p$ to be the path given by $\oshd_k p (t) = \frac{1}{k} \oshd p(t)$. Similarly, we let  $\shd_k p = \textup{Im} \oshd p$.  Note that if $p(t) \in V(g)$, then the $i$-th coordinate of $\oshd p(t)$ is the number of times (counted with direction) in the interval $[0,t]$ that $p$ passes through the edge $d_i$.

To the path $p$, we also associate a measure $\drk p$ on $\shd p$ by the rule $\drk^l p = (\oshd p)_* \lambda$, where $\lambda$ is the Lebesgue measure on $[0,1]$. We call this measure \emph{the darkness of the shadow of p}. Similarly, we define $\drk_k^l p = (\oshd_k)_* \lambda$. Note that if $l$ is the constant function $1$, then given $x \in \BR^N$, we have that $$\drk^l p \big[B_1(x)\big] = \frac{\# \textup{ times } p \textup{ passes through }x }{\textup{total path length of } p}$$

\nid When there is no chance for confusion, we will omit the $l$ superscript, and use $\drk p, \drk_k p$.

Now, Let $G = \mathcal{R}_n$ be the graph given by the join of $n$ circles. Label the $n$ edges with the labels $a_1, \ldots, a_n$. Given any length function, $l: E(G) \to \BR$, we can associate to each element $w \in F_n$ a unique path $p_w: [0,1] \to G$ that travels at constant speed and traverses the edges given by the word $w$. We will often confuse the path $p_w$ with the word $w$, and write  $\shd w, \drk w$, etc.  instead of $\shd p_w, \drk p_w$. \\

\nid \textbf{Example.} Consider the following subword of $F_2$:  $w = ab^3a^{-2}b^{-1}a^2$. The figure below shows the path $\oshd w$. The image of this path is $\shd w$. The measure $\drk w$ the sum of $1$-dimensional Lebesgue measures supported on each of the segments.

$$ \includegraphics[width=30mm]{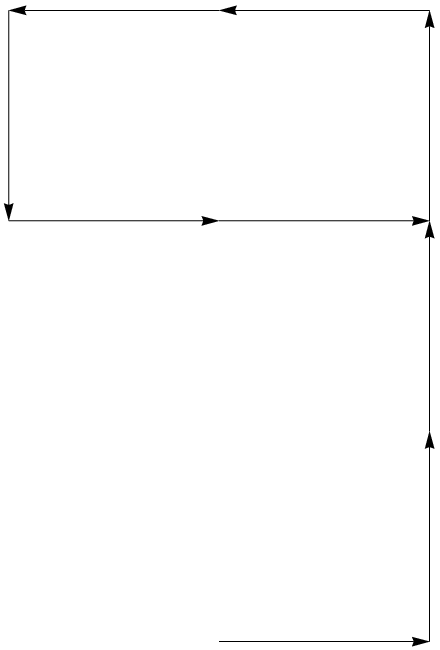} $$

\subsection{The main theorems.}

Recall that an element $f \in \aut$ is called \emph{fully irreducible} if there are no proper $f^i$-invariant free factors of $F_n$ for any $i \geq 1$.
\begin{theorem}\textbf{The Shape Of The Shadow.}\label{shadowshape} Let $f \in \aut$ be a fully irreducible automorphism such that $f_{ab}$ has finite order. Then there exists a convex polytope $\shd_{\infty} f \subset \BR^n$ with rational vertices, such that for any $x \in F_n$ with infinite $f$-orbit, it's true that
$$\lim_{k\to \infty} \shd_k f^k(x) = \shd_{\infty} f $$
where the the convergence above is in the Hausdorff topology.  Furthermore, if  $\phi \in \aut$, and $g = f^\phi$ then $\shd_{\infty} g = \phi_{ab} \shd_{\infty} f$.
\end{theorem}

\nid Notice that the last statement in the Theorem addresses the question of how the limiting polytope is affected by changing the generating set for $F_n$, since this corresponds to conjugating the automorphism $f$. Thus, Theorem \ref{shadowshape} canonically associates to $f$ a $GL_n(\BZ)$ orbit of a convex polytope with rational vertices.

We call the pair $(\shd_\infty, \drk_\infty f)$ the attracting homological lamination of $f$.

\begin{theorem} \textbf{The Darkness Of The Shadow.} \label{darkshadow}
Suppose $f$ is fully irreducible, and $f_{ab}$ has finite order. Let $\rho = \rho(f)$ be the dilatation of $f$. Then there exists a measure $\drk_{\infty}f$ supported at a point $\overline{\drk_{\infty}f} \in \BQ[\rho]^n$ such that for any length function $l: E(\mathcal{R}_n) \to \BR_{+}$, and for any $x \in F_n$ with infinite $f$-orbit it's true that:
$$\lim_{k \to \infty} \drk_k f^k(x) =  \drk_{\infty} f$$

\nid where the above convergence is in the weak-* topology. Furthermore, if $\phi \in \aut$, and $g = f^\phi$ then $\overline{\drk_{\infty} g} = \phi_{ab} \overline{\drk_{\infty} f}$.
\end{theorem}

\nid Thus, Theorem \ref{darkshadow} canonically associates to $f$ a $GL_n(\BZ)$ orbit of a point in $\BQ[\rho]^n$.

In \cite{BeFeH} Bestvina, Feighn and Handel define a notion of a stable lamination $\mathcal{L}$ for a fully irreducible automorphism $f \in \aut$, together with a topology in which given any $x \in F_n$ with infinite $f$-orbit,  the sequence  $\{f^i(x) \}_{i=1}^{\infty}$ converges to such a lamination. In analogy to this, we call the set $\shd_{\infty} f$, together with the measure $\drk_{\infty} f$ the \emph{homological shadow} of $\mathcal{L}$.

Throughout the proof, we will use the fact that $f$ is fully irreducible in only one way - that $f$ has a train-track representative with Perron-Frobenius transition matrix (see below for definitions). The class of automorphisms that have these properties is strictly larger than fully irreducible elements, and contains for example automorphisms induced by pseudo-Anosov diffeomorphisms of surfaces with many boundary components. We summarize this by stating the following theorem. 

\begin{theorem} Theorems \ref{shadowshape}, \ref{darkshadow} remain true if we replace the condition that $f$ projects to a fully irreducible element of $\textup{Out}(F_n)$ to the condition that $f$ projects to an element with a train track representative and a Perron-Frobenius transition matrix. 
\end{theorem}

\subsection{Expanding on the analogy to laminations}
 The analogy between laminations and the objects provided by Theorems \ref{shadowshape}, \ref{darkshadow} can be made more precise, by considering the following situation.

Suppose $\Sigma = \Sigma_{g,1}$ is a surface of genus $g$ with one boundary component, and let $p$ be a point on this component. Suppose $f \in \textup{Mod}(\Sigma)$ is a pseudo-Anosov mapping class that fixes the boundary component pointwise. The map $f$ induces a fully irreducible automorphism of the free group $\pi_1(\Sigma,p) \cong F_{2g}$. Suppose $\gamma$ is a non-peripheral closed curve on $\Sigma$. In \cite{flp} it is shown that in the appropriate topology, the forward orbit $\{f^i(\gamma)\}_{i=1}^\infty$ converges to an object called a stable (or attracting) projective measured lamination, $\mathcal{L}$.

Given any other curve $\delta$ on $\Sigma$, it is possible to define its intersection number with $\mathcal{L}$, which we denote $h_{\mathcal{L}}(\delta) = i(\delta, \mathcal{L})$. Let $\mathcal{C}$ be the set of simple closed non-peripheral curves in $\Sigma$. Consider the functions $h_i: \mathcal{C} \to \BR$ given by $h_i(\delta) = i(f^i(\gamma), \delta)$, where $i(\cdot, \cdot)$ is the geometric intersection number. Then there is a sequence of positive numbers $\{\lambda_i\}_{i=1}^{\infty}$ such that for any $\delta \in \mathcal{C}$ it's true that:

$$\lim_{i \to \infty} \lambda_i h_i(\delta) = h_{\mathcal{L}} (\delta) $$

A similar situation holds in our case. Let $\wh{i}$ be the oriented intersection form on $\mathcal{C}$. The form $\wh{i}$ is given by a symplectic form $\omega$ on $H_1(\Sigma,\BR)$. Suppose that $f$ is a pseudo Anosov mapping class as above, such that $f_{ab}$ has finite order. Let $x\in F_{2g}$ have infinite $f$-orbit. Given any word $w \in F_{2g}$, and a number $j$, let $\lceil w \rceil_j$ be the subword of $w$ obtained by taking the first through $j$-th letters of $w$ (and set $\lceil w \rceil_0 = e$). Given a number $t \in [0,1]$, let $r_i(t) = \lfloor it \rfloor$. Define a function $h_i: \mathcal{C} \times [0,1] \to \BR$ by setting:

$$h_i (\delta, t) = \wh{i} \big(\lceil f^{i}(x)\rceil_{r_i(t)}, \delta\big)$$

Now fix a curve $\delta$, and let $\wh{\delta}$ be its image in $H_1(\Sigma, \BR)$. Consider the measures $\sigma_i = \big(\frac{1}{i}h_{i}(\delta, t)\big)_* \lambda$, where $\lambda$ is the Lebesgue measure on $[0,1]$. Let $R_i = \textup{Range} \big[\frac{1}{i}h_i(t,\delta) \big]$. Then Theorems \ref{shadowshape}, \ref{darkshadow} immediately give:

\begin{corollary} In the notation above we have that:
\begin{enumerate}
\item In the Hausdorff topology:
$$\lim_{i \to \infty} R_i = \omega(\wh{\delta}, \shd_{\infty} f) $$
\item Let $\sigma_{\infty}$ be a probability measure concentrated at the point $\omega(\wh{\delta}, \overline{\drk_{\infty}f})$. Then $$\lim_{i \to \infty} \sigma_i = \sigma_\infty$$
\nid where the above convergence is in the weak-* topology.
\end{enumerate}
\end{corollary}

\nid Thus, if the attracting projective measured lamination $\mathcal{L}$ controls the geometric intersection numbers of a curve $\delta$ with iterates $f^i(x)$, then its homological shadow controls algebraic intersection numbers of $\delta$ with subwords of $f^i(x)$ (see \cite{Zor}, in which Zorich takes a different perspective on this problem). \\

\paragraph{\textbf{Acknowledgements.}} The author would like to thank Yair Minsky for very valuable suggestions at the inception of the project. He would also like to thank  Yael Algom-Kfir, Jayadev Athreya, Danny Calegari, Benson Farb, Thomas Koberda, Gregory Margulis, and  Igor Rivin for helpful conversations.

\section{Background and Notation}
\subsection{Graphs.} In the proof we will use various graphs. We allow both directed and undirected graphs, as well as allowing multiple edges between vertices, and edges between a vertex and itself. Given a graph $G$, we will always denote the vertex set of $G$ by $V(G)$, and its edge set by $E(G)$. If  $G$ is directed, and $e \in E(G)$, we denote its initial vertex by $\iota(e)$, and its terminal vertex by $\tau(e)$. If the graph $G$ is directed, we will denote by $\underline{G}$ its underlying non-directed graph, that is - the graph we get by forgetting the directions on each edge.

 Given a graph directed graph $G$, a \emph{path in G of length k} is a function $p: \{1, \ldots, k\} \to E(G)$ such that for each $1 \leq i \leq k-1$, it's true that $\iota(i+1) = \tau(i)$. In this case, we write $\iota(p) = \iota(1)$, $\tau(p) = \tau(k)$. The definition of a path is similar for a non-directed graph, except that every edge can be traversed in either direction. A path $p$ is called a circuit if $\iota(p) = \tau(p)$. We denote by $\mP_k(G)$ the set of paths of length $k$ in $G$, and by $\mP(G)$ the set of all paths in $G$. Similarly, given $v\in V(G)$, we define $\mP_k(G,v)$ to be the set of all paths of length $k$ whose initial vertex is $v$, and $\mP(G,v)$ to be the set of all paths whose initial vertex is $v$. Similarly, we define $\mathcal{C}_k(G)$, $\mathcal{C}(G)$ to be the sets of circuits of length $k$ in $G$, and the set of circuits in $G$ respectively.

 We will often also think of a path p of length $k$ as continuous function from $[0,1]$ to the geometric realization of $G$. We call such a path \emph{immersed} if the function $p$ is locally injective.

 Given a field $K$  (which we will always take to be $\BR$), let $KG$ be the abstract $K$ vector space spanned by the set $E(G)$. We have isomorphisms $KG \cong K^{E(G)} \cong K^{|E(G)|}$. Choose a positive orientation on each of the edges of $G$ (if $G$ is directed, we simply take the positive direction to be given by the direction in $G$). Let $\wh{\cdot}: \mP(G) \to KG$ given by $\wh{p} = \sum_{e \in E(G)} n_{e}e $ where $n_e$ is the number of times $p$ crosses $e$ in the positive direction minus the number of times it crosses it in the negative direction.

Notice that if we view $\underline{G}$ as a $1$-dimensional simplicial complex, we have a natural identification: $$\mC \cong H_1(\underline{G}, K)$$ \nid and an inclusion $\mC \hookrightarrow KG$.

\subsection{Adjacency graphs, transition matrices and Perron Frobenius matrices}
\nid A matrix $A \in M_n(\BR_{\geq 0})$ is called a \emph{Perron Frobenius Matrix} if $\exists k > 0$ such that $A^k \in M_n(\BR_{> 0})$. The following is a theorem of Perron and Frobenius, which we use extensively.
\begin{theorem}
Let $A$ be a Perron Frobenius matrix. Then $A$ has a real eigenvalue $\lambda > 0$ (called the Perron Frobenius eigenvalue) of multiplicity $1$, such that for any other eigenvalue $\mu$, it's true that $\lambda > |\mu|$. Furthermore, the coordinates of the eigenvalue corresponding to $\lambda$ are all positive.
\end{theorem}

 \nid Given a graph $G$, let $A(G)$ be its adjacency matrix, that is - if $V(G) = \{v_1, \ldots, v_k \}$, then $A(G)_{i,j}$ is the number of edges connecting $v_i$ to $v_j$. It is a standard fact that for any $k$, $(A(G))^{k}_{i,j}$ is the number of paths in $G$ whose initial vertex is $v_i$ and whose terminal vertex is $v_j$.  We call  a graph \emph{Perron Frobenius} if $A(G)$ is Perron Frobenius. An equivalent way of stating this definition is that $G$ is Perron Frobenius if there is a number $N > 0$ such that for any $k > N$ and for any $v,w \in V(G)$, there exists a path $p$ of length $k$ such that $\iota(p) = v$ and $\tau(p) = w$.

Suppose $G$ is undirected, and $E(G) = \{e_1, \ldots, e_N \}$. To any  continuous map $\phi: G \to G$ that sends paths to paths, we can associate a matrix $T\in M_N(\BR)$, called the \emph{transition matrix of $\phi$} in the following way.  Set $T_{i,j}$ to be the number of times $\phi(e_i)$ passes through $e_j$ (in either direction). The map $\phi$ is called Perron Frobenius if the transition matrix $T$ is Perron Frobenius.
\subsection{Fully irreducible automorphisms and their train track representatives. } Given a graph $G$, and a vertex $v \in V(G)$, we have $\pi_1(\underline{G},v) \cong F$ where $F$ is a free group.  Since $\underline{G}$ is a $K_{F,1}$ space,  any element $\phi \in \textup{Aut}(F)$ is induced by a continuous function $\phi: G \to G$. Furthermore, given an edge $e \in E(G)$, we have that $f(e)$ is a path in $G$. \\

A subgroup $H < F$ is called a \emph{free factor} if there is a subgroup $K < F$ such that $F = H * K$. An element $\phi \in \textup{Aut}(F)$ is called \emph{fully irreducible} if for every $i > 0$, there is no free factor $H$ such that $f(H)$ is conjugate to $H$.

 Given $f \in \aut$, a \emph{train track representative of f} is a triple $(G, \phi, \sim_v)$ of a graph, map, and a collection of equivalence relations with the following properties.
\begin{enumerate}
\item  The graph $G$ satisfies $\pi_1(G,v) \cong F_n$, for any $v \in G$.
\item The map $\phi: G \to G$ is a continuous map that sends paths to paths and induces the same \emph{outer automorphism} of $F_n$ as $f$.
\item The set $\sim = \{\sim_v \}_{v \in V(G)}$ is a collection of equivalence relations on the edges incident at each vertex $v$, such that $\sim_v$ has at least two equivlaence classes for each $v$. A path $p$ in $G$ is called \emph{legal} if no two of its concurrent edges are $\sim$-equivalent. The relationship $~$ satisfies the condition that given any legal path $p$, the path $\phi(p)$ is also legal. In particular, the path $\phi(p)$ is immersed.
\end{enumerate}
 If $f$ has a train track representative $(G,\phi,\sim)$, then by replacing $f$ with $f^K$ for some integer $K$ we can assume that $f$ fixes a vertex $v \in V(G)$. In this case, $f$ and $\phi$ will induce the same automorphism of $\pi_1(G,v) = F_n$. \\

\nid The following Theorem is due to work of Bestvina and Handel \cite{BeH}:

\begin{theorem}
Let $f \in \aut$ be fully irreducible. Then there exists at least one train track representative for $f$. Furthermore, there exists a number $\rho(f) > 1$ called the \emph{dilatation of f} such that for any train track representative $(G,\phi, \sim)$ the transition matrix $T$ of $\phi$ is Perron Frobenius with highest eigenvalue equal to $\rho$.
\end{theorem}

If we let $l$ be the eigenvector of $T$ corresponding to the eigenvalue $\rho$, then its coordinates give a function $l: E(G) \to \BR_+$. We call this length function the \emph{train track length function}. Since $l$ is a $\rho$-eigenvector, it follows that if $p \in \mP(G)$ then $l\big(\phi(p)\big) = \rho \cdot l\big( p\big)$.

\subsection{Bounded cancellation In free groups} In the proof of Theorems \ref{shadowshape} and \ref{darkshadow}, we will require the following theorem of Thurston (c.f. \cite{coop}).
\begin{theorem} \label{thurston} Let $B$ be a finite generating set for $F_n$, and let $|\cdot|$ be the word length function given by the generating set $B$. The for every $\phi \in\aut$, there exists a number $C > 0$, such that for every $\alpha, \beta \in F_n$ it's true that:
$$|\phi(\alpha)| + |\phi(\beta)| - C \leq |\phi(\alpha\beta)|  $$
\end{theorem}

\subsection{The Hausdorff Metric} Recall that given two compact sets $X, Y \subset \BR^N$, the \emph{Hausdorff distance} between $X$ and $Y$ is defined as:
$$d_H(X,Y) = \max \{\sup_{x \in X} d(x,Y), \sup_{y \in Y} d(y,X) \}$$
\nid where $d$ is the Euclidean metric. A related notation we will often use, given a compact set $X$ and a real number $R > 0$  is:
$$B_R(X) = \{x \in \BR^N \mid d(x,X) < R \} $$

\subsection{Weak-* convergence of measures} We say that a sequence of probability measures $\{\mu_k \}_{k =1}^{\infty}$ on $\BR^n$ weak-* converge to the measure $\mu$ if for every bounded, continuous function $h: \BR^n \to \BR$ it's true that:
$$\lim_{k \to \infty} \int_{\BR^n} h \mu_k = \int_{\BR^n} h \mu $$
One convenient criterion that we will use for weak-* convergence is the following.
\begin{lemma} Let  $\{\mu_k \}_{k =1}^{\infty}$ be a sequence of probability measures on $\BR^n$, and let $\mu$ be a probability measure supported at the point $Q \in \BR^n$. Suppose that for every closed ball $B \subset \BR^n$ it's true that $\lim_{k \to \infty} \mu_k \big[ B\big] = \mu[B]$. Then the measures $\mu_k$ weak-* converge to $\mu$.
\end{lemma}

\begin{proof}
Let $h: \BR^n \to \BR$ be a bounded continuous function. Let $M$ be an upper bound for $|h|$. Fix $\epsilon > 0$. Since $h$ is continuous, there exists a $\delta$ such that $|h(x) - h(y)| < \epsilon$ for all $x,y \in \overline{B_\delta(Q)}$. For all sufficiently large values of $k$ we have that $\mu_k(\overline{B_\delta(Q)}^{c})  < \frac{\epsilon}{M}$. For these values of $k$:

$$|\int_{\BR^n} h \mu_k - h(Q)| \leq | \int_{\overline{B_\delta(Q)}} h \mu_k - h(Q)| + |\int_{\overline{B_\delta(Q)}^{c}} h \mu_k|  \leq (1 - \epsilon)\epsilon + \epsilon $$
and the result follows.
\end{proof}

\section{Proofs of main Theorems}
\subsection{Proof of Theorem \ref{shadowshape}}

\subsubsection{Reduction to the case $f_{ab} = I_n$} During the course of the proof, we will often replace $f$ with some power of itself. To do this, we require the following Lemma.

\begin{lemma} \label{powers1} If $f\ \in \aut$ is fully irreducible such that $f_{ab}$ has finite order, and the conclusion of Theorem \ref{shadowshape} holds for $f^L$ for some  $L > 0$, then it holds  for $f$ as well.
\end{lemma}

\begin{proof}
Choose $x \in F_n$ for which the orbit of $x$ under $f$ is infinite. Then the orbit of $x$ under $f^L$ is infinite, as is the orbit of $f^r x$ under the action of $f^L$, for any $r \geq 1$. Write $g = f^L$. Given any integer $P$, we have that $P = qL + r$ for some $q,r \in \BN$, and

 $$\shd_P f^P x = \shd_P g^q \big(f^r x \big) = \frac{P}{q} \shd_q g^q f^r(x) $$

 The number $r$ can take on the values $0, \ldots, L -1$. For each particular value of $r$ it's true that:
 $$\lim_{q \to \infty} \shd_q g^q \big(f^r(x) \big) = \shd_{\infty} g $$

\nid Furthermore, $\lim_{P \to \infty} \frac{P}{q} = L$. Thus, it follows that:

 $$\lim_{P \to \infty} \shd_P f^P(x) = L \cdot \shd_{\infty} g$$

\end{proof}

\nid One immediate consequence of this Lemma is that we can (and will) assume henceforth that $f_{ab} = I_n$. Furthermore, whenever we have a graph $G$, and a map $\phi:G \to G$ that realized the same outer automorphism as $f$, we will always assume $\phi$ fixes a vertex of $G$ by replacing $f$ with a power of itself.

\subsubsection{Controlling the action on $\BR G$ for a general graph $G$}

\nid Let $G$ be a graph with $N$ edges, and $v_0 \in V(G)$ such that $\pi_1(G,v_0) \cong F_n$.  Let $\phi: (G,v_0) \to (G,v_0)$ be a continuous function that sends edges to paths, such that $\phi$ induces the automorphism $f$ on $F_n$. Choose a preferred orientation on each of the edges of $G$.

Let $\phi_{ab}$ be the action induced by $\phi$ on $\BR G$. Note that the fact that $f_{ab} = I_n$ does not imply that $\phi_{ab} = I_N$. However, the action of $\phi_{ab}$ on $\wh{\mP{G}}$ is relatively simple.  Indeed, for every $v \in V(G)$, choose a path $p_v \in \mP(G)$ that connects $v$ to $v_{0}$. Given any path $p$, let $\overline{p}$ be the path obtained by traversing $p$ backwards from $\tau(p)$ to $\iota(p)$. Given a path $p \in \mP(G)$ such that $\iota(p) = v$ and $\tau(p) = w$, we have that:
$$\phi_{ab} \wh{p} = \phi_{ab}\wh{p}  - \phi_{ab}\wh{p_v} + \phi_{ab}\wh{p_w} + \phi_{ab}\wh{p_v}  - \phi_{ab}\wh{p_w} = \phi_{ab}(\wh{\overline{p_v}p p_w}) - \phi_{ab}\wh{p_v}  + \phi_{ab}\wh{p_w} $$

Since $\overline{p_v} p p_w$ is a circuit, $\phi_{ab}$ acts trivially on it. Thus, setting $u = \wh{p_w} - \wh{p_v}$ we get:

$$\phi_{ab} \wh{p} =  \wh{p} + (u - \phi_{ab} u) $$

\nid Setting $R = \phi_{ab} - I_N$ we can write $\phi_{ab} \wh{p} = \wh{p} + R(\wh{p})$. The above argument shows that $R(\wh{p})$ depends only on the initial and terminal vertices of $p$, and thus can take on only finitely many values.

\subsubsection{Half-point sets of shadows and their iterates} Fix a graph $G$ as above. Call a point $y \in \BR^N$ \emph{path accessible} if there exists a path $p \in \mP(G)$ such that $y \in \shd(p)$. Let $\mathcal{A} \subset \BR^N$ be the set of path accessible points.  Let $$\mathcal{H} = \{x \in \frac{1}{2} \BZ^N | \textup{ exactly one coordinate of } x \textup{ is not an integer } \} \cap \mathcal{A}$$

\nid Given a path $p \in \mP(G)$, define the \emph{half point set of $p$}, or $HP(p)$,  to be the set $\shd(p) \cap \mathcal{H}$. The set $HP(p)$ has the property that $d_H (HP(p), \shd(p)) \leq \frac{1}{2}$. We will study the sets $HP(\phi^k(p))$. These are relatively simple to study because of the following lemma, that shows that given a path $p$, the set $HP(\phi(p))$ can be determined solely  by considering $HP(p)$.

\begin{lemma} \label{Locality}  There exists a function $L: \mathcal{H} \to 2^\mathcal{H}$ such that for any path $p \in \mP(G)$:
$$HP(\phi (p)) = \bigcup_{x \in HP(p)} L(x) $$
\end{lemma}

\begin{proof}
\nid Each element $x \in \mathcal{H}$ has a preferred coordinate - the nonintegral one. This coordinate corresponds to an element in $E(G)$ which we denote $e(x)$. If $x \in HP(p)$, then $p$ passes through the edge $e(x)$ at least once, either in the positive or negative direction.

Given an edge $e \in E(G)$, let $p_e$ be the path that traverses $e$ in the positive direction.  Let $x \in \mathcal{H}$.  Define $\lfloor x \rfloor = x - \frac{1}{2} \wh{p_{e(x)}}$, and set:

$$L(x) = HP(\phi(p_{e(x)}))  +  \phi_{ab} \lfloor x \rfloor $$

\nid where the above expression is taken to mean the translation of the set $HP(\phi(p_{e(x)})) $ by the vector $\phi_{ab} \lfloor x \rfloor$.  Notice that if $p$, and $p = p_1 p_2$ then $$\shd p = \shd p_1 \bigcup \big(\shd p_2 + \phi_{ab} \wh{p_1} \big)$$

\nid  Fix $x \in \mathcal{H}$. We consider two cases. \\

\nid \emph{Case 1:} Let $p$ be a path of the form $p = p_1 p_{e(x)}$, such that the last segment of $\oshd p$ passes through $x$  then

$$\shd \phi(p) = \shd(p_1) \bigcup \big(\shd{\phi(p_{e(x)})} + \phi_{ab}(\wh{p_1}) \big)   = \shd(p_1) \bigcup \big(\shd{\phi(p_{e(x)})} + \phi_{ab}  \lfloor x \rfloor\big)$$

\nid So:
$$HP(p) = HP(p_1) \bigcup L\big(x \big) $$

\nid \emph{Case 2:} Now suppose  $p$ is a path of the form  $p = p_1 \overline{p_{e(x)}}$, such that the last edge of $\oshd p$ passes through $x$  then

$$\shd \phi(p) = \shd p_1 \bigcup \big(\shd{\phi(\overline{p_{e(x)}})} + \phi_{ab}(\wh{p_1}) \big) =  \shd p_1 \cup \big(\shd \overline{\phi(p_{e(x)})} + \phi_{ab}(\wh{p_1}) \big)$$

\nid Now notice that $\wh{p_1} = \lfloor x \rfloor + \wh{p_e}$.  Thus

$$\shd \phi(p) =  \shd p_1 \bigcup \big(\shd \overline{\phi(p_{e(x)})} + \phi_{ab}(\wh{p_e})  + \phi_{ab} \lfloor x\rfloor \big)$$

\nid Furthermore, we have that  $\shd  \overline{\phi(p_{e(x)})} + \phi_{ab}(\wh{p_e}) = \shd \phi (p_{e(x)})$. Indeed, as paths in $\BR^N$:  $$\oshd \overline{\phi(p)} + \phi_{ab}(\wh{p_e}) = \overline{\oshd \phi(p)}$$

\nid Both paths traverse the same edges in reverse order. Thus, once again in this case:
$$HP(\phi(p)) = HP(p_1) \bigcup L(x) $$

\nid The result now follows from induction on the length of the path $p$.

\end{proof}

\begin{lemma} \label{notdepend} For any $x_1, x_2 \in \mathcal{H}$ it's true that: $L(x_1) - x_1 = L(x_2) - x_2$.
\end{lemma}

\begin{proof}
 If $x \in \mathcal{H}$, there is a path $p_x \in \mP(G,v_0)$ such that $\lfloor x \rfloor \in \shd (p)$. Thus, $\phi_{ab} \lfloor x \rfloor = \lfloor x \rfloor + R(\lfloor x \rfloor)$, where $R(\lfloor x \rfloor)$ is depends only on the terminal vertex of $p_x$, and thus it is a function of $e(x)$. Thus, the set $L(x) - x$ depends only on $e(x)$.
\end{proof}

By Lemma \ref{notdepend}, there is a function $L': E(G) \to 2^{\BR^N}$ such that for any $x \in \mathcal{H}$, it's true that $L(x) - x = L' (e(x))$.

Now consider the following directed graph, which we call the \emph{half point graph of $\phi$} or $HP_\phi$. For every edge $e \in E(G)$, there is a vertex $v_e \in V(HP_\phi)$. Given an edges $d,e \in E(G)$, connect the vertex $v_d$ to $v_e$ with $m$ edges emanating from $v_d$, where $m$ is the number of times $\phi(d)$ passes through the edge $e$ (in either direction).

Let $\eta \in E(HP_{\phi})$. The edge $\eta$ associated to a segment in $\phi(d)$, for some $d \in E(G)$, which is in turn associated to a segment in $\shd \phi(d)$. Let $y_\eta$ be the half-point of this segment and let $x_{\eta} = \frac{1}{2} \wh{d}$. Define $H(\eta) = y_{\eta} - x_{\eta}$. Extend $H$ linearly to a function $H: \BR HP_{\phi} \to \BR^N$. \\

\nid \textbf{Example.} Let $G = \mathcal{R}_2$. Label the edges of $G$: $a$ and $b$. Let $\phi: G \to G$ be given by $\phi(a) = ababa^{-1}$, and $\phi(b) = bab^{-1}a^{-1}b$. The figures below show the parametrized shadows $\oshd \phi(a)$, $\oshd \phi(b)$.

$$ \includegraphics[width=30mm]{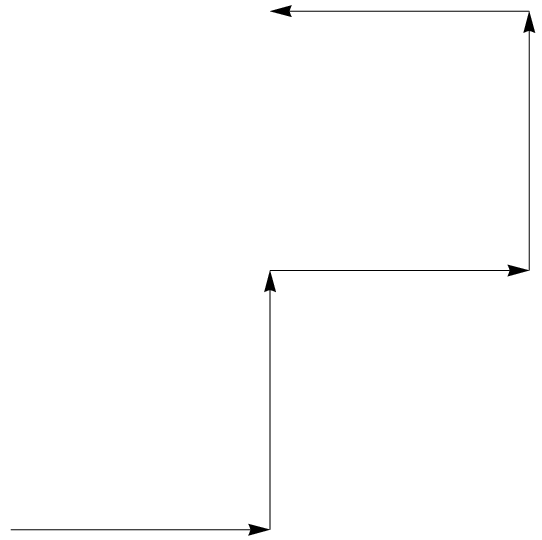} \includegraphics[width=20mm]{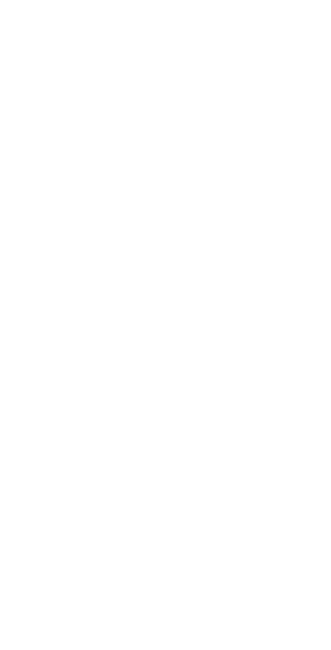}  \includegraphics[width=30mm]{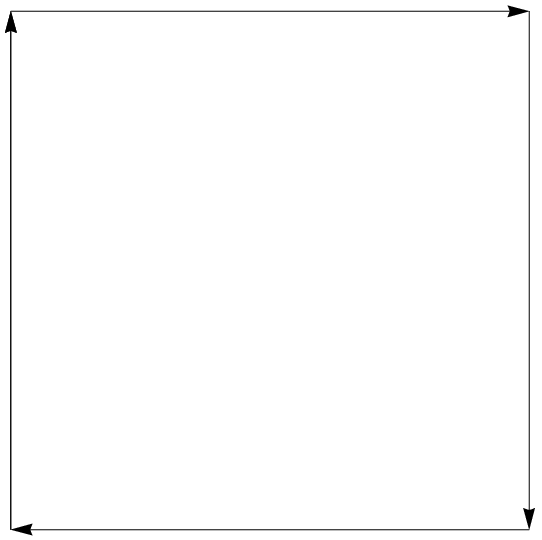} $$

The graph $HP_\phi$ has two vertices: $v_a, v_b$. The path $\phi(a)$ has length five. Three of its segments pass through the edge $a$ and two through $b$. Thus, there will be three edges from $v_a$ to $v_a$, whose $H$ values are: $\left(\begin{array}{c} 0 \\ 0 \end{array} \right), \left(\begin{array}{c} 1 \\ 1 \end{array} \right), \left(\begin{array}{c} 1 \\ 2 \end{array} \right)$ and two from $v_a$ to $v_b$ whose $H$ values are $\left(\begin{array}{c} \frac{1}{2} \\ \frac{1}{2} \end{array} \right), \left(\begin{array}{c} \frac{3}{2} \\ \frac{3}{2} \end{array} \right)$. Similarly, $v_b$ is connected by three edges two itself whose $H$ values are $\left(\begin{array}{c} 0 \\ 0 \end{array} \right), \left(\begin{array}{c} 0 \\ 0 \end{array} \right), \left(\begin{array}{c} 1 \\ 0 \end{array} \right)$ and by two to $v_a$, whose $H$ values are $\left(\begin{array}{c} \frac{1}{2} \\ \frac{1}{2} \end{array} \right), \left(\begin{array}{c} \frac{1}{2} \\ \frac{-1}{2} \end{array} \right)$.  \\

\nid Given a path $p \in \mP(G)$, the graph $HP_{\phi}$ can be used to find $HP\big(\phi^k(p)\big)$ for any value of $k$.

\begin{lemma} \label{pathsandhp} Let $p \in \mP(G)$. Write $p = p_1 \cdot \ldots \cdot p_s$ as a concatenation of edges, and let $e_1, \ldots, e_s$ be these edges. Let $P^k_i = \wh{\mP_k(HP_\phi, e_i)} \subset \BR HP_{\phi}$. Then for any $k \geq 1$:

$$HP(\phi^k (p)) = \bigcup_{i=1}^{s} \big[\wh{p_1 \cdots p_{i-1}} + H\big(P^k_i \big) \big] $$
\end{lemma}

\begin{proof}
\nid  Given $x \in HP(\phi^k(p))$, by Lemmas \ref{Locality} and \ref{notdepend}, we have that $x = x_{k-1} + y_{k-1}$ where $x_{k-1} \in HP(\phi^{k-1}(p))$ and $y_{k-1} \in L'(e(x_{k-1}))$. Similarly, $x_{k-1} = x_{k-2} + y_{k-2}$, where $x_{k-2} \in HP(\phi^{k-2}(p))$ and $y_{k-2} \in L'(e(x_{k-2}))$. Continuing in this manner, we can associate, perhaps non uniquely,  to $x$ two sequence $x_0, x_1, \ldots, x_{k-1}$, where $x_i \in HP(\phi^{i}(p))$ and $y_i \in L'(e(x_i))$. We have that  $y_i = H(\eta_i)$ for some edge $\eta_i \in HP_\phi$. Furthermore, by definition we must have that $\tau(\eta_i) = \iota(\eta_{i+1})$. The point $x$ is completely determined by $x_0$ and the sequence of $y_i$'s . The sequence of $y_i$'s determines a path $\pi$ of length $k$ in $\mP_k(HP_\phi,v_{e(x_0)})$, and their sum is $x_0 + H(\wh{\pi})$ . Similarly, any such path determines a sequence of $x_i$'s and $y_i$'s, and hence a point in $HP(\phi^k(p))$.

 \end{proof}

\nid The following lemma and proposition describe the set $\wh{\mP_k(HP_\phi, v_{e_i})}$ when the map $\phi$ is Perron Frobenius.

\begin{lemma} \label{perronfrob} In the notation above, if $\phi$ is Perron Frobenius then $HP_\phi$ is Perron Frobenius.
\end{lemma}

\begin{proof}
Let $A$ be the adjacency matrix of $HP_\phi$. Given an integer $k$, and a pair of indices $i,j$ we have that $\big(A^k\big)_{i,j} > 0$ if and only if there is a path of length $k$ from $v_{e_i}$ to $v_{e_j}$. By lemma \ref{pathsandhp}, this happens if and only if $\phi^k(e_i)$ passes through $e_j$. Since $\phi$ is Perron Frobenius, this happens for all sufficiently large values of the number  $k$.
\end{proof}

\begin{prop}\label{shapeofpaths} Let $\Gamma$ be a finite, Perron Frobenius directed graph with $N$ edges. Then there exists a convex polytope $\Pi \subset \BR^N$  with rational vertices such that for any $v_0 \in V(\Gamma)$: $$ \lim_{k \to \infty} \frac{1}{k} \wh{\mP(\Gamma,v_0)} = \Pi$$ where the above limit is in the Hausdorff metric. \end{prop}
\begin{proof}

  Fix a vertex $v_0 \in V(\Gamma)$. Choose an ordering $e_1, \ldots, e_N$ of $E(\Gamma)$. This gives an identification of $\BR \Gamma$ with $\BR^N$. Given an real number $k > 0$, let:

 $$\Sigma_k = \mC \cap \{\left(\begin{array}{c} x_1\\ \dots \\ x_{N} \end{array} \right) \in \BR^N \mid x_1 + \ldots + x_{N} = k, \textup{ and  } \forall i: x_i \geq 0 \} $$

 For $k_1 \leq k_2$ let $\Sigma_{[k_1,k_2]} = \bigcup_{k \in [k_1,k_2]}\Sigma_{k}$. Let $\Sigma_{[k_1, k_2]}^\BZ =  \Sigma_{[k_1, k_2]} \cap \BZ^{N}$. Since the graph $\Gamma$ it is Perron Frobenius, it is strongly connected, that is: each point can be connected by a path to each other point. This means that there is a constant $C_1$ such that any path can be extended to a circuit in $\Gamma$ containing $v_0$  by adding at most $C_1$ (positively oriented) edges. Thus, $\wh{\mP_k(\Gamma,v_0)} \subset \Sigma^\BZ_{[k,k+C_1]}$. This gives:

 $$\wh{\mP_k(\Gamma,v_0)} \subset B_{C_1}\big(\Sigma_{[k, k+C_1]}\big)$$
 \nid By definition, we have $\frac{1}{k} \Sigma_k = \Sigma_1$, so:
 $$\frac{1}{k}\wh{\mP_k(\Gamma,v_0)}  \subset  B_{\frac{C_1}{k}}\big(\Sigma_{[1, 1 + \frac{C_1}{k}]}\big)$$

  Let $t >  1$. Since $\Sigma_t = t \Sigma_1$, we have that any point in $y \in \Sigma_t$ is of the form $tp$ for some $p \in \Sigma_1$, and thus:
 $$d(y,\Sigma_1) \leq d(p,tp) = (t-1)\|p\|_2 \leq (t -1)\max_{q \in \Sigma_1} \|q\|_2 $$

 Setting $D = 1 + \max_{q \in \Sigma_1} \|q\|_2 $, we get $\Sigma_{[1, 1 + \frac{C_1}{k}]} \subset B_{\frac{DC_1}{k}}\big(\Sigma_1 \big)$:

 $$\frac{1}{k}\wh{\mP_k(\Gamma,v_0)} \subset B_{\frac{2DC_1}{k} }\big(\Sigma_1\big)$$

    For $x = \sum_{i = 1}^{N}a_i e_i \in \BR \Gamma$, let $\textup{supp}(x) = \{e_i \mid a_i \neq 0\}$. Let $k \in \BN$ and fix some $x \in \Sigma_k^\BZ$. The set  $\textup{supp}(x)$ can be viewed as a subgraph of $\underline{\Gamma}$ in an obvious way. We claim that if $\textup{supp}(x)$ is connected, then $x \in \wh{\mathcal{P}_{k}}$.

To see this, note that since $x \in \mC$, there exists \emph{circuits} $\gamma_1, \ldots, \gamma_l$ in $\underline{\Gamma}$ such that $\sum_{i = 1}^{l} \wh{\gamma_i} = x$. If $c_1, c_2$ are two intersecting circuits in $\underline{\Gamma}$, then there is a circuit $c_3$ in $\underline{\Gamma}$ such that $\wh{c_3} = \wh{c_1} + \wh{c_2}$. This circuit is given by following $c_1$ until it first intersects $c_2$, then following all of $c_2$, and finally resuming $c_1$. Since the graph $\textup{supp}(x)$ is connected, reorder the $\gamma_i$'s such that for all $i > 2$, $\gamma_i$ intersects $\bigcup_{j < i} \gamma_j$. Thus, there is a circuit $\gamma$ in $\underline{\Gamma}$ such that $\wh{\gamma} = x$.

 Since $\gamma$ a circuit in $\underline{\Gamma}$, it may contain edges whose direction is opposite the direction given by $\Gamma$. Suppose, without loss of generality that $ e_1 = (v,w) \in E(\Gamma)$, and $(w,v)$ is an edge in $\gamma$.  Since $x \in \Sigma_k$, we have that $a_1 \geq 0$. Therefore, $e_1$ is also an edge in $\gamma$. Since $\gamma$ is a circuit, we can cyclically reorder $\gamma$ to get a circuit :

 $$\delta = (v,w) \gamma_1' (w,v) \gamma_2'$$ where $\gamma_1'$, $\gamma_2'$ are circuits in $\underline{\Gamma}$. We have that $\wh{\delta} = \wh{\gamma_1'} + \wh{\gamma_2'}= x$. The circuits $\gamma_1'$ and $\gamma_2'$ must intersect, because their union equals the support of $x$, which is connected. Using the procedure described above, we find a circuit $\gamma'$ in $\underline{\Gamma}$ such that $\wh{\gamma'} = x$. Furthermore, $\gamma'$ traverses the same edges as $\gamma$ the same number of times, except for $e_1$ which it traverses one less time in each direction, so its length is $2$ less than the length of $\gamma$. We can proceed inductively in this manner, until we get a circuit $\gamma_0$ which traverses all edges in the direction given by $\Gamma$, such that $\wh{\gamma_0} = x$. \\

\nid For any $k$ let $$\Delta_{k} = \{x \in \Sigma_{k} \mid \textup{supp}(x) \textup{ is  not connected } \}$$

\nid The above discussion gives us that $$\Sigma_k^\BZ - \Delta_k \subset \wh{\mP_k(\Gamma)} $$
\noindent The space $\mC$ is a subspace of $\BR \Gamma$ that is defined over $\BZ$, and thus $\mC \cap \BZ^N$ is a lattice of rank $\dim \mC$ in $\mC$. Therefore, there exists a constant $C_2$ such that for any $k$, it's true that $d_H(\Sigma_k,\Sigma_k^\BZ) < C_2$, so we have:
$$\Sigma_k - \Delta_k \subset B_{C_2}\left( \wh{\mP_k(\Gamma)}\right) \subset B_{C_1 + C_2}\left(\wh{\mP_k(\Gamma,v_0)}\right)$$
If $\textup{supp}(x)$ is not connected, then it is a proper, disconnected subgraph of $\Gamma$.  Let $\Gamma_1, \Gamma_2$ be two connected components of  $\textup{supp}(x)$. For $i = 1,2$, choose vertices $v_i \in V(\Gamma_i)$. Since $\Gamma$ is strongly connected, there are paths $P_{1,2}$ and $P_{2,1}$ in $\Gamma$ connecting $v_1$ to $v_2$ and $v_2$ to $v_1$ respectively. Since $\Gamma$ is directed, we can't have that $\wh{P_{1,2}} = - \wh{P_{2,1}}$. The cycle $P_{1,2} P_{2,1}$ has nontrivial image in $\BR \Gamma$, and is not contained in $\textup{span } \mathcal{C}\big(\underline{\textup{supp}(x)}\big)$.  Thus $\textup{span }\wh{\mathcal{C}(\underline{\textup{supp}(x)})}$ is a proper subspace of $\mC$. Since $\Sigma_k$ is a $\dim \mC - 1$ dimensional polytope that does not pass through the origin, we that $\Sigma_k \cap \textup{span } \wh{\mathcal{C}(\underline{\textup{supp}(x)})}$ is a polytope of dimension $\leq \dim \mC - 2$. Thus $ \Sigma_k - \textup{span }\wh{\mathcal{C}(\underline{\textup{supp}(x)})}$ is dense in $\Sigma_k$. Since $$\Delta_k \subset \bigcup_{\Gamma' \textup{ is a disconnected subgraph of $\Gamma$}} \textup{span }\wh{\mathcal{C}(\underline{\Gamma'})}$$
\nid and the above union is finite, we have that $\Sigma_k - \Delta_k$ is dense in $\Sigma_k$. Thus:

$$\Sigma_k \subset B_{C_1 + C_2}\left(\wh{\mP_k(\Gamma,v_0)}\right)$$
\nid And:
$$\Sigma_1 \subset B_{\frac{C_1 + C_2}{k}}\left(\frac{1}{k} \wh{\mP_k(\Gamma,v_0)}\right) $$
Putting the above inclusions together,  we get:

$$ \Sigma_1 \subset B_{\frac{C_1 + C_2}{k}} \big( \frac{1}{k}\wh{\mP_k(\Gamma, v_0)} \big) \subset B_{\frac{C_1 + C_2 + 2DC_1}{k}} \big(\Sigma_1 \big) $$
and thus $$\lim_{k \to \infty} \frac{1}{k} \wh{\mP_k(\Gamma,v_0)}  = \Sigma_1$$

  Since $\Sigma_1$ is the intersection of a convex polytope (a simplex) with a subspace, it is a convex polytope itself.  Since $\Sigma_1$ was defined by equations and inequalities with coefficients in $\BZ$, it has rational vertices.
 \end{proof}

\begin{corollary} \label{shadmap} In the notation above, if the transition matrix of $\phi$ is Perron Frobenius then there exists a convex polytope $\shd_{\infty} \phi \subset \BR G$ with rational vertices such that for any path $p \in \mP(G,v_0)$ we have that:
$$\shd_{\infty} \phi = \lim_{k \to \infty} \shd_k(\phi^k (p)) $$

\end{corollary}

\begin{proof}
Set  $ \shd_{\infty} \phi = H\big( \Sigma_1 \big)$,  where $\Sigma_1 \subset \BR HP_{\phi}$ is the set defined in the proof of Proposition \ref{shapeofpaths}. For convexity, note that this is just the image of a convex polytope under a linear map. For the rationality of the vertices, note that $H$ is given by a matrix with integer entries, and that $\Sigma_1$ has rational coefficients. The convergence follows directly from Proposition \ref{shapeofpaths}, lemmas \ref{pathsandhp} and \ref{perronfrob} and the fact that for any path $p$, it's true that $d_H(\shd p, HP(p)) \leq \frac{1}{2}$.
\end{proof}

\subsubsection{Passing from shadows in $\BR G$ to shadows in $\BR^n$} Let $G, \phi$ be as above. Let $T$ be a minimal spanning tree for $G$. Let $A = E(G) - E(T)$. There is a  map $R_A: \mP(G,v_0) \to F_A \cong F_n$, (where $F_A$ is the free group on the set $A$) given by reading off the edges in $A$ that $p$ passes through, and obtaining a reduced word in $F_A$ (c.f. \ref{}).

A path $p \in \mP(G)$ is called \emph{reduced} if it immersed, that is: it does not contain any subpaths of the form $e \overline{e}$, where $e \in E(G)$. Given any path $p$  consider the following path: start with $p$, and repeatedly removing all subpaths of the form $e \overline{e}$, until there are none left to remove. We call this new path $p^{red}$. We define $\shd^{red}(p) = \shd(p^{red})$ and call this the \emph{reduced shadow of p}.

 If $p \in \mathcal{P}(G,v_0) $, we have that $R_A(p) = R_A(p^{red})$. Let $R_A: \BR G \to F_A$ be the projection map that reads off the coordinates corresponding to edges in $A$. If $R_A(p) = w$  then we have that $R_A(\shd^{red} p) = \shd w$.

Let $F_A$ be the free group on the elements of $A$.  The map $\phi$ induces an automorphism $f_A$ of $F_A$. If we identify $F_A$ with $F_n$, then this automorphism is conjugate to $f$.

\begin{prop} Let $G,\phi$ be a train track representative for $f$. Then, for any word $w \in F_A$ with infinite $f_A$-orbit, we have that $$\lim_{k \to \infty} \shd_k f^k(w) = R_A \shd_{\infty} \phi $$

\end{prop}

\begin{proof}
Suppose $p$ is a legal path then $\phi^k(p)^{red} = \phi^k{p}$. Thus, if $p$ is legal: $$R_A\big(\shd_k \phi^k(p) \big) = \shd_k f_A^k\big(R_A(p)\big)  $$

 \nid Since $\phi$ is a train track representative, the transition matrix for $\phi$ is Perron Frobenius. Thus. by Corollary \ref{shadmap}:

 $$\lim_{k\to \infty } \shd_k f_A^k(R_A(p)) = R_A \big(\shd_{\infty} \phi \big) $$

\nid and $R_A \big(\shd_{\infty} \phi \big)$ is a convex polytope with rational vertices.

Now suppose that $p \in \mP(G,v_0)$ is a general, not necessarily legal such that the forward orbit of $p$ under the action of $\phi$ is infinite. An illegal subpath of $p$ of length $2$ is called an \emph{illegal turn}. Let $I(p)$ be the number of illegal turns in $p$. Since paths of length $1$ are legal, as are all  their images under $\phi^k$ for every $k$, we get that $I\big(\phi^k(p)^{red} \big)\geq I(p)$ for every $k$. Since the length of $\phi^k(p)$ goes to infinity, we get that given any $L > 0$, the paths $\phi^k(p)^{red}$ will contain a legal subpath of length $> L$ for all sufficiently large $k$.

  By \cite{BeFeH}, if there exists a number $C>0$ such that if $p$ contains a legal subpath $p'$ of length greater than $C$, then $p'$ contains a subpath $p''$ such that $\phi^k(p'')$ is a subpath of $\phi^k(p)^{red}$, for any $k$.  Replace $f$ with a sufficiently high power of itself, such that $\phi(p)^{red}$ hs a legal subpath of  length greater than $C$. Since $\phi$ is Perron Frobenius, by taking an even higher power we can assure that the path $p''$ passes through $v_0$. By replacing $p''$ with a legal subpath of itself, we can take $\iota(p'') = v_0$.

Write $q = \phi(p) = xp''y$ where $x,y$ are paths. Then for any $k$: $$\wh{x} + \shd \phi^k(p'') \subset \shd^{red} \phi^k(q)$$

this gives:

$$\frac{k}{k-1}\big(\frac{\wh{x}}{k-1} + \shd_{k-1} \phi^{k-1}(q) \big)\subseteq \frac{k}{k-1} \shd_{k-1}^{red} \phi^{k-1} (q)= \shd_k^{red} \phi^k(p) \subseteq \shd_k \phi^k(p) $$

 Since the leftmost and rightmost expressions above converge to $\shd_{\infty} \phi$, we have that:

$$\lim_{k\to \infty}  \shd^{red}_k \phi^k(p) = \shd_{\infty} \phi$$

\nid as long as $p$ has an infinite $\phi$ orbit and therefore $$\lim_{k\to \infty} \shd_k f_A^k (R(p)) = R_A \shd_{\infty} \phi $$
\nid as long as $R(p)$ has infinite $f_A$ orbit. The result now follows.
\end{proof}

\subsubsection{Passing from $f_A$ shadows to $f$ shadows}
Suppose $f = g f_A g^{-1}$ for $g \in \aut$. Suppose $w$ has infinite $f$ orbit. Then $g^{-1}(w)$ will have infinite $f_A$ orbit. The sets $\shd_k f_A^k(g^{-1}w)$ converge to $R_A \shd_{\infty} \phi$. To conclude the proof of Theorem \ref{shadowshape}, we need to describe $\shd_k gf_A^k(g^{-1}w)$

\begin{lemma} \label{conjugation}
Let $h \in \aut$. There exists a number $C = C(h)$ such that for any word $w \in F_n$, it's true that $d_{H}\big(\shd h(w), h_{ab} \shd w \big) < C$.
\end{lemma}
\begin{proof}
 Let $C_1 = C_1(h)$ be a number as provided by Theorem \ref{thurston} for the automorphism $h$ . Let $w \in F_n$, and $x \in \shd w \cap \BZ^n$. Write $w = w_1 w_2$, as a reduced product of words such that $\wh{w_1} = \tau (\oshd w_1) = x$. Let $y = h_{ab} x = \tau \big(\oshd h(w_1) \big)$. It might be true that $y \notin \shd h(w)$. However, by Theorem \ref{thurston}, $\exists y' \in \shd h(w)$ such that $d(y,y') < C_1$, or in other words: $$h_{ab} \shd w \subset B_{C_1} (\shd h(w)) $$

 \nid Let $C_2 = C_2(h^{-1})$ be a number as provided by Theorem \ref{thurston}, for the automorphism $h^{-1}$.  By the argument above,
$h_{ab}^{-1}\shd h(w) \subset B_{C_2} \big(\shd w\big)  $ so:

$$\shd h(w) \subseteq h_{ab} B_{C_2} \big( \shd w\big) \subseteq B_{\|h_{ab} \| C_2} \big(h_{ab} \shd(w) \big) $$

\nid where $\|h_{ab} \|$ is the operator norm of $h_{ab}$. The result now follows by setting $C = \max \{C_1, \|h_{ab}\| C_2 \}$.

\end{proof}

\nid We now have that $d_H\big(\shd_k gf_A^k (g^{-1}w), g_{ab} \shd_k f_A^k (g^{-1}w)\big) \leq \frac{C}{k}$. Theorem \ref{shadowshape} follows immediately, where $\shd_{\infty} f = g_{ab} \circ R_A \shd_{\infty} \phi$.

\subsection{Proof of Theorem \ref{darkshadow}}

As in the proof of Theorem \ref{shadowshape}, we begin by noting that we can replace $f$ with a power of itself. The proof of the following lemma is nearly identical to the proof of Lemma \ref{powers1}, and we do not include it.

\begin{lemma} \label{powers2} Let $f \in \aut$ be fully irreducible such that $f_{ab}$ has finite order. Suppose that there exists an integer $N$ such that the conclusion of Theorem \ref{darkshadow} holds for $f^N$, then it holds for $f$ as well.

\end{lemma}

\subsubsection{Half-point darkness measures} Let $G,\phi$ be a train track representative graph with $N$ edges, with $\pi_1(G,v_0) \cong F_n$, and let $\phi: (G,v_0) \to (G,v_0)$ be a continuous function that induces $f$. Assume further that $\phi$ is Perron Frobenius.  Let $l: E(G) \to \BR_{> 0}$ be a length function. Given a path $p \in \mP(G)$, the measure $\drk p$ is supported on segments of the form $[x, x + e_i]$ where $x \in \BZ^N$, and $e_i$ is a standard basis vector. Let $\sigma$ be such a segment. The restriction of $\drk p$ to $\sigma$ is a $1$-dimensional Lebesgue measure. Each such $\sigma$ corresponds to an edge $e_{\sigma} \in E(G)$. Let $c_{\sigma}(p)$ be the number of times $p$ traverses $\sigma$ (in either direction). If we let $l(p)$ be the total length of $p$ then:

$$\drk p \big[\sigma\big] = \frac{c_{\sigma}(p) l(e_{\sigma})}{l(p)}$$

The measure $\drk p$ is completely determined by evaluating it on all segments $\sigma$ as described above. We define a new measure $\drk HP(p)$ on $\mathcal{H}$, by setting, for every $x \in \mathcal{H}$: $\drk HP(p) \big[x \big] = \drk p \big[ \sigma_x \big] $ where $\sigma_x$ is the segment that passes through $x$. For any integer $k$, we define similarly a measure $\drk_k HP(p)$ on $\frac{1}{k} \mathcal{H}$. The following lemma allows us to prove work with half-point measures.

\begin{lemma} Suppose $\{p_k\}_{k=1}^{\infty}$ is a sequence in $\mP(G)$ such $\bigcup_{k=1}^{\infty} \shd_k p_k$ is a bounded subset of $\BR^N$,  and:  $$\lim_{k \to \infty} \drk_k HP(p_k) = \drk_\infty$$ in the weak-* topology, where $\drk_{\infty}$ is a measure supported at a single point.  Then: $$\lim_{k \to \infty} \drk_k p_k = \drk_{\infty} $$

\end{lemma}

\begin{proof}
Let $h: \BR^N \to \BR$ be a bounded continuous function. Let $Q = \mathbb{E} (\drk_{\infty})$. We have that $$\lim_{k \to \infty} \int_{\BR^N} h \drk_k HP(p_k) = h(Q) $$
The supports of all the measures $\drk_k p_k$, and $\drk_k HP(p_k)$ lie in some compact set $X \subset \BR^N$. By absolute continuity, for every $\epsilon > 0$, there is a $\delta > 0$ such that whenever $y \in \BR^N$, and $\sigma$ is a segment containing $y$, of length less than $\delta$, it's true that $|\int_{\sigma}h ds - h(y)  | < \epsilon \cdot \textup{length}(\sigma)$ where $ds$ is a Lebesgue measure on $\sigma$ of total measure $\textup{length}(\sigma)$.

Thus, fixing a value of $\epsilon > 0$, if $\frac{1}{k} < \delta$ and $p$ is any path such that $\shd_k p \subset X$, we have  that $$|\int_{\BR^N} h \drk_k p - \int_{\BR^N} h \drk_k HP(p)| < \epsilon$$

\nid Since $\epsilon$ was arbitrary, we get $\lim_{k \to \infty} \int_{\BR^N} h \drk_k p_k = h(Q)$, as required.
\end{proof}

Notice that to apply this Lemma, we will need the fact that for any path $p$, the set $\bigcup_{k=1}^{\infty} \shd_k \phi^k(p)$ is bounded. This follows immediately from Corollary \ref{shadmap}.

Fix a train track representative $G, \phi$ for $f$. Suppose that $G$ has $N$ edges, and let $l: E(G) \to \BR_+$ be any length function.
Let $p \in \mP(G)$, and  write $p = p_1 \ldots p_s$ as a concatenations of directed edges. Suppose that the edges $p$ passes through are $e_1, \ldots, e_s$. Suppose that the graph $HP_\phi$ has $M$ edges. Let $X = \biguplus \BR HP_{\phi}$. For every $1 \leq i \leq M$, let $X^i$ denote the $i$-th copy of $\BR HP_{\phi}$. Extend the function $H$ defined in section 3.1.3 to a function $H_p: X \to \BR^N$ by setting  by setting: $H_p|_{X_i} = \wh{p_1 \cdots p_{i -1}} + H$. Define functions  $c_k$ on $X$ by setting, for $x \in X^i$: $$c_k(x) = \#\{q \in \mP_k(HP_\phi, p_i) | \wh{q} = x\}$$
Given such an $x$, the point $H(x)$ is the a half point in $\BR^N$, corresponding to the edge $e(x) \in E(G)$. Define $l(x) = l\big( e(x)\big)$.  Define a measure $\nu_k$ on $X$ by setting, for any $x \in X$:

$$\nu_k(x) = \frac{c_k(x) l(x)}{l(\phi^k(p))} $$
Following the definitions, we get $\drk HP \big(\phi^k(p) \big) = H_{p*} \nu_k$.

\subsubsection{Train track length functions and random walks on graphs}
Now fix $l: E(G) \to \BR_+$ be the train track length function. For this function, the measures $\nu_k$ are simpler to describe. Fix a path $p = p_1 \cdot \ldots \cdot p_s$. For any $k$, and any path $q$, we have that  $l\big(\phi^k(q)\big) = \rho^k l \big( q\big)$. Thus:

$$\nu_k = \sum_{i = 1}^{s} \frac{l\big(\phi^k(p_i)\big)}{l\big(\phi^k(p) \big)}\big(H_p|_{X^i} \big)_* \drk HP \big(\phi^k(p_i) \big) =  \sum_{i = 1}^{s} \frac{l\big(p_i\big)}{l\big(p \big)}\big(H_p|_{X^i} \big)_* \drk HP \big(\phi^k(p_i) \big)$$

 Thus $\nu_k$ is a convex combination of the measures  $\big(H_p|_{X^i} \big)_* \drk HP \big(\phi^k(p_i) \big)$. We will study each of these separately. Henceforth, assume $s = 1$, so $p = p_1$, and $p_1$ corresponds to the edge $\eta \in E(G)$.

 \begin{prop} \label{halfpointconverge1}Let $G,\phi, l$ be as above.  Then there exists a point $\overline{\drk_{\infty} \phi} \in \BQ[\rho(\phi)]^N$ such that for any path $p = p_1$ of length $1$, we have that  $$\lim_{k\to \infty} \drk_k \phi^k p = \drk_\infty \phi $$
 in the weak-* topology, where $\drk_\infty \phi$ is a point measure supported at $\overline{\drk_\infty \phi}$.
 \end{prop}

 \begin{proof}
 Given vertices $v_d, v_e \in E(HP_\phi)$,  each edge $h$ connecting $v_d$ to $v_e$ corresponds to a half point in $x \in HP(\phi(d))$ such that $e(x) = e$. Let $\mu(h) = \frac{l(e)}{l(\phi(d))} $

\nid Notice that $\mu$ defines a probability measure on the outgoing edges of each vertex in $HP_\phi$. Given an integer $k$, we can use $\mu$ to define the random walk  measure $\mu_k$ on $\mP_k(HP_\phi)$. This is given by $\mu_k(h_1 \ldots h_k) = \prod_{i=1}^k \mu(h_i)$. Suppose $h_i$ connects the vertex $v_{e_i}$ to the vertex $v_{e_{i+1}}$ . Since $\phi$ multiplies lengths of all paths by $\rho$, we have that:

$$\mu_k(h_1 \ldots h_k) = \frac{\prod_{i=1}^k l(e_{i+1})}{\prod_{i=1}^k l \rho (e_{i})} = \frac{l(e_{k+1})}{\rho^k l(e_1)} = \frac{l(e_{k+1})}{l(\phi^k(e_1))}$$

\nid Let $\mu_k^{res}$ be the probability measure $\mu_k$ restricted to the set $\mP_k(HP_\phi, v_{\eta})$. From the above equation, we get by definition that $\nu_k = \big(\wh{\cdot}\big)_* \mu_k^{res}$, where  $\textup{ } \wh{\cdot} \textup{ }$  is viewed as a function from $\mP_k(HP_\phi,v_{\eta})$ to $\BZ^M$.

Let $\mathfrak{X}$ be the set of bi-infinite indexed paths in $HP_\phi$, that is: $\mathfrak{X}$ is the set of all functions $p: \BZ \to E(HP_\phi)$ such that for each $i \in \BZ$, we have that $p(i) p(i+1)$ is a path in $HP_\phi$. A \emph{cylindrical} set in $\mathfrak{X}$  is a set of the form $C(\underline{i},\underline{e}) = \{p \in \mathfrak{X} | p(i_j) = e_j, \forall 1 \leq j \leq l  \} $, where $\underline{i} = (i_1, \ldots, i_l)$ is an increasing sequence of length $l$ ,in $\BZ$ and $\underline{e} = (e_1, \ldots, e_l) \in E(HP_\phi)^l$. Let $\Gamma$ be the $\sigma$-algebra generated by all cylindrical sets.

Define a measure $\mu$ on $\mathfrak{X}$  on sets in $\Gamma$ in the following way. Let $C(\underline{i},\underline{e})$  as above. Let $k = i_l - i_1$. Let $$\mu\big(C(\underline{i},\underline{e})\big) = \mu_k\big(\{p \in \mP_k(HP_\phi) | p(i_j - i_i) = e_j, \forall 1 \leq j \leq l  \} \big) $$

It is a standard fact that such a definition gives rise to a well define measure of finite mass on $\mathfrak{X}$. Let $s: \BZ \to \BZ$ be given by $s(x) = x+1$. Let  $T: \mathfrak{X} \to \mathfrak{X}$ be given by $T(p) = p \circ s$. It is clear that the measure $\mu$ is $T$ invariant, and is standard that $T$ is ergodic with respect to $\mu$.

Given an edge $e \in E(HP_\phi)$, and a number $i \in \BZ$, let $\chi_i^e = 1_{C(e,i)}$. Notice that $\chi_i^e = f^e_0 \circ T^i $. Let $\Psi_k: \mathfrak{X} \to \BR^M$ be given by $$\Psi_k(p) = \big( \sum_{i = 1}^k \chi_i^e \big)_{e \in E(HP_\phi)} (p) =  \big( \sum_{i = 1}^k \chi_0^e \circ T^i \big)_{e \in E(HP_\phi)} (p)$$ Given a path $p' \in \mP_k(HP_\phi)$, choose a path $p \in \mathfrak{X}$ such that for every $1 \leq i \leq k$, it's true that $p'(i) = p(i)$. We have that: $\wh{p'} = \Psi_k(p)$.

Furthermore, if we set $\mu^{res}$ to be the measure $\mu$ restriced to the set $\{p \in \mathfrak{X} | \iota(p(1)) = v_{\eta} \}$, we have by definition that $\nu_k = \Psi_{k*} \mu^{res}$.

Now, let $g_k: \BR^M \to \BR^M$ be given by $g_k(x) = \frac{x}{k}$, and  $\widetilde{\nu_k} = (g_k)_*\nu_k$. Then we have that:

$$\drk_k HP\big(\phi^k(p_1)\big) = H_*\widetilde{ \nu_k} =  \big[H_p \circ \Psi_k \big]_* \mu^{res} $$

\nid By Birkhoff's ergodic theorem, the functions $\frac{1}{k} \Psi_k$ converge almost surely as $k \to \infty$ to a $T$-invariant function $\Psi_\infty$. Since $T$ is ergodic, this function is constant $\mu$ almost everywhere. Call this constant $Q$. Since $\mu^{res}$ is absolutely continuous with respect to $\mu$, this function is equal to $Q$,  $\mu^{res}$ almost everywhere. This implies that for any $\epsilon > 0$ and $0 < C < 1$, it's true that for all sufficiently large value of $k$:  $\mu^{res} \big[\frac{1}{k} \Psi_k^{-1} B_\epsilon(Q) \big] > C$, and thus $\tilde{\nu_k} \big[B_\epsilon(Q) \big] > C$. Similarly, for $P \neq Q$, for all sufficiently large values of $k$ it's true that $\tilde{\nu_k} \big[B_\epsilon(P) \big]< C$. This immediately implies that the measures $\tilde{\nu_k}$ weak-* converge to a point measure supported at $Q$, and thus the measures $\drk_k HP\big(\phi^k(p_1)\big)$ weak-* converge to a point measure supported at $H(Q)$.

To calculate the point $Q$, notice that for every $i$, and every $e \in E(HP_\phi)$:
$\mathbb{E}(\chi^e_i) = \mu(e)$. Thus:
$$\mathbb{E}\big(\frac{1}{k} \Psi_k) = \frac{1}{k} \sum_{i=1}^k \mu(e) = \mu(e) $$

So $\Psi_\infty = \big(\mu(e) \big)_{e \in E(HP_\phi)}$, almost everywhere. By definition, $\mu(e)$ is in the field generated by $\rho$ and the lengths of the edges in $G$. Since these lengths are simply coordinates of a $\rho$-eigenvector of a matrix with coefficients in $\BZ$, they belong to the field $\BQ[\rho]$, as required.

\end{proof}

Since the measure provided by Proposition \ref{halfpointconverge1} doesn't depend on the choice of the initial path $p$, we get the following.

\begin{corollary} \label{halfpointconverge2} Let $G,\phi, l$ be as above.  Then there exists a point $\overline{\drk_{\infty} \phi} \in \BQ[\rho(\phi)]^N$ such that for any path $p$ of length $1$, we have that  $$\lim_{k\to \infty} \drk_k \phi^k p = \drk_\infty \phi $$
 in the weak-* topology where $\drk_\infty \phi$ is a point measure supported at $\overline{\drk_\infty \phi}$.

\end{corollary}

We generalize this corollary, by allowing the length function $l$ to vary.

\begin{lemma} \label{changelengthfunction} Let $G,\phi, l, \drk_{\infty} \phi$ be as above. Let $l'$ be any length function on $E(G)$. Then for any path $p$:
$$\lim_{k\to \infty} \drk^{l'}_k \phi^k p = \drk_\infty \phi $$
\end{lemma}

\begin{proof} There is a constant $C > 0$ such that for any path $q \in \mP(G)$, we have that $$\frac{1}{C}l'(q) \leq l(p) \leq Cl'(q)$$
If we set $\nu_k$ to be the measure described above with respect to $l$, and $\nu_k'$ to be the same measure with respect to $l'$, we have that for every point x:
$$\frac{1}{C^2} \nu_k'(x) \leq \nu_k (x) \leq C^2 \nu_k(x)$$

\nid This immediately implies the result.

\end{proof}

\subsubsection{Switching from the train track graph to $\mathcal{R}_n$.} Let $T$, $A$,  $R_A$, and $f_A$ be the tree, set, map and automorphism defined in section 3.1.4.  We can view $l$ as a length function on $\mathcal{R}_n$.  Extend $l$  to a length function $l: E(G) \to \BR_+$.

\nid Given any immersed loop $p \in \mP(G,v_0)$, we have that $l \big(R_A(p) \big) >0$. Write $\drk p = \drk^A p + \drk^T p$, where $ \drk^A p $ is supported on shadows of edges in $A$, and $ \drk^T p$ is supported on shadows of edges in $T$. Define, similarly, $\drk^A_k p$ and $\drk^T_k p$ Let $m^A(p), m^T(p)$ be the total masses of $ \drk^A p $, and $ \drk^T p $ respectively.   By definition: $$\drk R_A(p) = \frac{1}{m^A(p)} R_{A*} \drk^A p  $$

Now add the assumptions that $p$ is legal. For every edge $e \in E(G)$, let $n_e$ be the number of times $p$ traverses this edge (in either direction.)  For any $k$ we have that:

$$\drk_k R_A\big( \phi^k(p) \big) =\frac{1}{m^A(\phi^k p)} R_{A*}\big[ \drk^A_k \phi^k(p) \big]   $$

Let $\Theta$ be the transition matrix for $\phi$, and let $v$ be its Perron Frobenius eigenvector. Each coordinate of $v = \big(v_e \big)_{e \in E(G)}$ is positive. Let $\| v\|_A = \sum_{e \in A} |v_e|$. Notice that $\|v\|_A \neq 0$.  Let $\mu^A = \sum_{e \in A} l(e)$, $\mu = \sum_{e \in E(G)} l(e)$. For any $k$, the coordinates of the vector $\Theta^k (n_e)_{e \in E(G)}$ give the number of times $\phi^k(p)$ passes through each edge of $E(G)$.   For all sufficiently high values of $k$,  the vectors $\Theta^k (n_e)_{e \in E(G)}$ are in the positive orthant, and hence the sequence $\{\Theta^k (n_e)_{e \in E(G)}\}_{k=1}^{\infty}$ converges projectively to $v$. Thus:

$$\lim_{k \to \infty} m^A (\phi^k(p)) = \frac{\mu^A \cdot \| v\|_A}{\mu \cdot \|v\|_1} > 0 $$

Thus, $\exists C > 0$ such that for all sufficiently large $k$: $m^A(\phi^k(p)) > C$. Now, since $\drk_k \phi^k(p) = \drk^A_k \phi^k(p) + \drk^T_k \phi^k(p)$, and $\lim_{k \to \infty} \drk_k \phi^k(p) = \drk_{\infty} \phi$ is a measure supported at a point, we must have that the measures $\drk^A \phi^k(p)$ converge to a point measure supported at the same point. Thus,

$$\lim_{k \to \infty} \drk_k f_A^k(p) = R_{A*} \drk_{\infty} \phi $$

and the right hand side of the above equality is a measure supported at a point in $\BQ [\rho]$. Now suppose that $p$ has infinite $\phi$ orbit, but remove the assumption that it is legal. Notice that now, $\drk_k f_A^k(p) = R_{A*} \drk_k^A [\phi^k(p)]^{red}$.  Replacing $p$ with $(\phi^M p)^{red}$ for a sufficiently large $M$, we get that $p$ contains arbitrarily long legal subpaths. Thus, we can assume that $p$ contains a legal subpath $q$ that is a circuit, such that $\phi^k(q)$ is a subpath of $\phi^k(p)^{red}$ for any $k$.
Given any path $\pi \in \mP(G)$, a subpath $\pi'$  of $\pi^{red}$ and a set $B \subset \BR^N$,   we have that:

$$ \label{estimate} \frac{l(\pi')}{l(\pi^{red})}\drk \pi' \big[ B\big] \leq \drk \pi^{red}  \big[ B\big] \leq \frac{l(\pi)}{l(\pi^{red})} \drk \pi \big[B \big]  $$

\nid in particular, for any $k$:

$$\frac{l(\phi^{k} q)}{l\big((\phi^k p)^{red}\big)} \drk_k \phi^k q \big[B \big] \leq \drk_k (\phi^k p)^{red} \big[ B \big] \leq  \frac{l \big(\phi^k p\big)}{l\big((\phi^k p)^{red}\big)} \drk_k \phi^k p \big[ B\big]$$

\begin{lemma} \label{ratios} In the notation above, there exists a constant $C > 0$ such that for every $k$: it's true that $$\frac{l(\phi^{k} q)}{l\big((\phi^k p)^{red}\big)} \geq C, \frac{l \big(\phi^k p\big)}{l\big((\phi^k p)^{red}\big)} \geq C$$

\end{lemma}
\begin{proof}

Suppose first that $l$ is the train length function.  In this case: $l\big( \phi^k p \big) = \rho^k l\big(p\big)$, and $l\big( \phi^k q \big) = \rho^k l\big(q\big)$. Furthermore, we have that $$l\big( \phi^k q \big) \leq \big((\phi^k p)^{red} \big) \leq l\big( \phi^k p \big)$$
and the result follows immediately for this case. Now suppose we choose a different length function, $l'$. Since the metrics given by $l'$ and $l$ are bi-Lipschitz equivalent, we that there exists a constant $D>0$ such that for every path $r$, it's true that $\frac{1}{D}l'(r) \leq l(r) \leq D l'(r)$ and thus:
$$\frac{l'(\phi^{k} q)}{l'\big((\phi^k p)^{red}\big)} \geq D^2 \frac{l(\phi^{k} q)}{l\big((\phi^k p)^{red}\big)} $$
\nid and
$$\frac{l' \big(\phi^k p\big)}{l' \big((\phi^k p)^{red}\big)} \geq D^2 \frac{l \big(\phi^k p\big)}{l\big((\phi^k p)^{red}\big)} $$

 \nid and the result now follows for the general case.
\end{proof}

 Now, since $\lim_{k\to \infty} \drk_k \phi^k q = \lim_{k \to \infty} \drk_k \phi^k p = \drk_{\infty} \phi$, Lemma \ref{ratios} gives that for any open ball $B$, it's true that:
$$\frac{1}{C} \drk_k \phi^k q \big[B \big] \leq \drk_k (\phi^k p)^{red} \big[ B \big] \leq C \drk_k \phi^k p \big[ B\big]$$
and thus, in the weak-* topology: $\lim_{k \to \infty} \drk_k (\phi^{k})^{red}= \drk_{\infty} \phi$.
Since $m^A \big(\drk^A_k \phi^k(q)  \big)  \leq m^A \big(\drk^A_k (\phi^k p)^{red}  \big) $, then the same argument as the legal case gives that $$\lim_{k\to \infty} \drk_{k} f_A^k\big(R_A(p)\big) = \lim_{k\to \infty} R_{A*} \drk^A_k (\phi^k p)^{red} = R_{A*} \drk_{\infty} \phi$$

\nid We denote the point $\mathbb{E}\big(R_{A*} \drk_{\infty} \phi \big) = \overline{\drk_{\infty} f_A}$. Note that this point does not depend on the length function $l$.
\subsubsection{Behavior under conjugation.} We've now proved Theorem \ref{darkshadow} for the autmorphism $f_A$, which is conjugate to $f$. We want to prove it for $f$, as well as proving it more generally for $hfh^{-1}$, where $h \in \aut$. The following Proposition will complete the proof.

\begin{prop} Given the notation above, and any $h \in \aut$,  any length function $l$ on $E(\mathcal{R}_n)$, and any word $w$ with infinite $f$-orbit:  $$\lim_{k \to \infty} \drk_{k} hf^kh^{-1}(w) = h_{ab*} \drk_{\infty} f  $$
\end{prop}

\begin{proof}
\nid Let $C_1 = C_1(h)$ be a number as provided by Theorem \ref{thurston} for the automorphism $h$.  Let $C_2 = \frac{l_M}{l_m}$ where $l_M$ is the length of the longest edge of $\mathcal{R}_n$ and $l_m$ is the length of the shortest. Let $C_3$ be a number such that for any $x \in F_n$ we have that $\frac{1}{C_3} \| x\| \leq \|hx \| \leq C_3 \|x \|$, where $\| . \|$ is the word length metric on $F_n$. Let $C = \max \{C_1, C_2, C_3 \}$.

 Given any $v \in F_n$, and any edge $e$ of $\shd v$, corresponding to the edge $d$ of $\mathcal{R}_n$, we have that $\drk v \big[ e \big] = \frac{l(d) n_e}{l(v)}$, where $n_e$ is the number of times $\oshd v$ passes through $e$. Similarly to the argument in Lemma \ref{conjugation}, if $\oshd v$ passes through $e$, then $\oshd gv$ passes through at least one edge in $B_C( h_{ab} e)$, which will have at length at least $\frac{1}{C} l(d)$. Furthermore, by the same reasoning, if $\oshd v$ passes $j$ times through $e$, any two of which are separated by at least $C$ edges, then $\oshd hv$ passes at least $j$ times through edges in $B_C(h_{ab} e)$, each of which has length at most $C l(e)$ . Also notice that $l(hv) < Cl(v)$. Thus:

$$\drk hv \big[ B_C(h_{ab} e)\big] \geq \frac{\frac{1}{C} l(d) \lfloor \frac{n_e}{C}\rfloor  }{Cl(v)} \geq \frac{1}{2C^3} \drk v \big[ e \big] $$

 \nid By replacing $C$ by a larger constant, we may assume that it is a small cancellation constant and a quasi-isometry constant for $h^{-1}$ as well. Let $R>0$ be a sufficiently large number such that $h_{ab}^{-1} B_C(0) \subset B_R(0)$. Then, by the same reasoning:

 $$\drk v \big[B_{C + R} (e) \big] \geq \frac{1}{2C^3} \drk hv \big[B_C(g_{ab} e) \big]  $$

\nid Thus, for any open ball $X \subset \BR^N$ we have that $$\frac{1}{2C^3}h_{ab*} \drk v \big[X\big] \leq \drk hv \big[ X\big] \leq 2C^3 h_{ab*} \drk v \big[B_{R+C}(X) \big]  $$

\nid And thus, for any $k$: $$\frac{1}{2C^3}h_{ab*} \drk_k v \big[X\big] \leq \drk hv \big[ X\big] \leq 2C^3 h_{ab*} \drk_k v \big[B_{\frac{R+C}{k}}(X) \big]  $$

\nid Writing: $ \drk_k h f_A^k f^{-1} w = \drk h\big( f_A^k (h^{-1}w )\big) $, and using the fact that $\lim_{k\to \infty}\drk_k  f_A^k (h^{-1}w ) = \drk_{\infty} f$, the above inequality implies that if $\drk_{\infty} f \big[ X \big] = 0$ (that is - $h_{ab}^{-1}X$ does not contain the center of mass of $f$) then $\lim_{k \to \infty}   \drk_k h f_A^k f^{-1} w \big[ X\big] = 0$. Since all of the above measures are probability measures, this implies  $$\lim_{k\to \infty} \drk_k h f_A^k h^{-1} w \big[ h_{ab}\overline{\drk_{\infty} \phi} \big] = 1$$ \nid Therfore, $\lim_{k \to \infty} \drk_k h f_A^k f^{-1} w = h_{ab*} \drk_{\infty} f$, as required.

\end{proof}


\vspace{2mm}

\begin{thebibliography}{99}
\bibitem{BeFeH} M. Bestvina, M. Feighn, and H. Handel. Laminations, trees, and irreducible automorphisms of free groups. {\it Geom. Func. Anal.} \textbf{7} (1997), no. 2, 215--244, Erratum: \textbf{7} (1997) no. 6, 1143.
\bibitem{BeH} M. Bestvina, and H. Handel. Train tracks and automorphisms of free groups. {\it Ann. Math.} 135 (1995) 1--51.
\bibitem{flp}A. Fathi and F. Laudenbach and V. Po\'enaru.  {\it Travaux de Thurston sur les Surfaces}.  Soc. Math. de France, Paris,  Ast\'erisque 66-67,  1979.
\bibitem{coop} D. Cooper. Automorphisms of free groups have finitely generated fixed point sets. {\it J. Alg.} 111 (1987) 453--456.
\bibitem{Zor} A. Zorich. How do the leaves of a closed 1-form wind around a surface. In the collection: "Pseudoperiodic Topology" {\it AMS Translations}, Ser. 2, vol. 197, AMS, Providence, RI, 135--178 (1999)


\end{thebibliography}
\end{document}